\documentclass[11pt,oneside,reqno]{amsart}
\usepackage{graphicx}
\usepackage{amssymb}
\usepackage{amsmath}
\usepackage{amsthm}
\usepackage{amscd}
\usepackage{xcolor}
\usepackage{enumerate}
\usepackage{array}
\usepackage{caption}
\usepackage{subcaption}
\usepackage{siunitx}
\usepackage{textcomp}
\usepackage[T1]{fontenc}

\RequirePackage{dsfont} 
 \scrollmode
\allowdisplaybreaks

\newcolumntype{L}[1]{>{\raggedright\let\newline\\\arraybackslash\hspace{0pt}}m{#1}}
\newcolumntype{C}[1]{>{\centering\let\newline\\\arraybackslash\hspace{0pt}}m{#1}}
\newcolumntype{R}[1]{>{\raggedleft\let\newline\\\arraybackslash\hspace{0pt}}m{#1}}

\newtheorem{thm}{Theorem}[section]

\newtheorem{lem}[thm]{Lemma}

\newtheorem{prop}[thm]{Proposition}
\newtheorem{rem}[thm]{Remark}
\newtheorem{fact}[thm]{Fact}
\theoremstyle{definition}

\newcommand{\N}{\mathbb{Z}_+}

\newcommand{\R}{\mathbb{R}}

\newcommand{\comp}{\textrm{comp}}
\newcommand{\sgn}{\textrm{sgn}}

\makeatletter
\newcommand*{\defeq}{\mathrel{\rlap{%
                     \raisebox{0.3ex}{$\m@th\cdot$}}%
                     \raisebox{-0.3ex}{$\m@th\cdot$}}%
                     =}
\let\c@table\c@figure
\makeatother

\newcommand{\pic}[4]
{
	\begin{figure}[!htbp]
		\begin{center}
			\includegraphics[width=#2\textwidth]{#1}
			\begin{minipage}[c]{0.8\textwidth}
				\begin{center}
					\caption{#3}\label{#4}
				\end{center}
			\end{minipage}
		\end{center}
	\end{figure}
}

\numberwithin{equation}{section}
\title[Euler scheme for approximation of solution of nonlinear ODEs]{Euler scheme for approximation of solution of nonlinear ODEs under inexact information}

\author[Natalia Czy\.z{}ewska]{Natalia Czy\.z{}ewska}
\author[Pawe\l \ M. Morkisz]{Pawe\l \ M. Morkisz}
\author[Pawe\l \ Przyby\l owicz]{Pawe\l \ Przyby\l owicz}
\address[Czy\.z{}ewska]{AGH University, Faculty of Applied Mathematics, al. Mickiewicza 30, 30-059 Krak\'ow, Poland}
\email{nczyzew@agh.edu.pl, corresponding author}
\address[Morkisz]{AGH University, Faculty of Applied Mathematics, al. Mickiewicza 30, 30-059 Krak\'ow, Poland}
\email{morkiszp@agh.edu.pl}
\address[Przyby\l owicz]{AGH University, Faculty of Applied Mathematics, al. Mickiewicza 30, 30-059 Krak\'ow, Poland}
\email{pprzybyl@agh.edu.pl}
\begin{document}
\begin{abstract}
We investigate error of the Euler scheme in the case when the right-hand side function of the underlying ODE satisfies nonstandard assumptions such as local one-sided Lipschitz condition and local H\"older continuity. Moreover, we assume two cases in regards to information availability: exact and noisy with respect to the right-hand side function. Optimality analysis of the Euler scheme is also provided. Finally, we present the results of some numerical experiments.
\end{abstract}
\maketitle	

\tableofcontents

\section{Introduction}

In this paper we investigate error of the Euler scheme for the following initial-value problems
\begin{eqnarray}
\label{ode_gen}
	\left\{ \begin{array}{ll}
	z'(t)=g(t,z(t)), \quad &t\in [a, b],\\
	z(a)=\xi,
	\end{array} \right.
\end{eqnarray}
where $-\infty < a < b < + \infty$, $\xi\in\R^d$ and a right-hand side function $g: [a, b]\times\R^d\to\R^d$, $d\in\N$. Comparing to the classical assumptions such as global Lipschitz condition we assume that the right-hand side function $g$ is continuous and of at most linear growth, but it satisfies the one-sided Lipschitz condition together with the H\"older condition only locally. 

The numerical approximation of solutions of ordinary differential equations (ODEs) has drawn the attention of numerous scientists \cite{butcher, kincaid, driver, griff, hawan, jackiewicz}. Furthermore, solving ODE problems with noisy information was studied in \cite{TBPP2022, TBMGPMPP2021, Kac2, KaPl90, KaPr16, MoPl16, MoPl20, NOV, Pla96, Pla14}. Also, approximation of the solutions of partial differential equations or stochastic differential equations with both exact and inexact information was a subject of extensive study, \cite{ames, larss, mazum, samar, Wer96, Wer97} and \cite{kloeden, milstein, KaMoPr19, PMPP17, PMPP19} respectively.
Nevertheless, when it comes to the nonstandard assumptions about the right-hand side function of the initial-value problem,
only the existence and uniqueness of the solutions were examined, see \cite{gorn_ing}. There is a gap in the literature about numerical methods for such cases under nonstandard assumptions.

If we take a closer look at the computation on a digital computer, we quickly conclude that the computer uses only a finite set of numbers because of the finite memory. So all computer calculations are burdened with errors, e.g., from rounding.
Also, the novel CPUs and GPUs have an option to speed up the computations when there is lower precision of the processed numbers. From the mathematical point of view, lowering the precision can be interpreted as adding noisy information with the control for the relative error of the noise. 
Hence, all the applied numerical methods should have a strict error analysis with noisy information setting.

In this paper, we deal with nonstandard assumptions about the right-hand side function $g$, and we allow both exact and inexact information about the initial value $\xi$ and values of the right-hand side function $g$. We changed a bit of assumptions from the auxiliary ordinary differential equation in \cite{CZPMPP2022} where we assume locally
H\"older continuity and global one-sided Lipschitz conditions. Now we locally consider
H\"older continuity and one-sided Lipschitz condition only in a certain neighborhood.
This weakening of the assumptions was caused by the necessity that came from a phase change model of metallic materials \cite{CZIEEE2022, CZPMPP2022}.

According to our best knowledge, there is a lack of literature on numerical methods of initial-value problems under nonstandard assumptions with noisy information. We focus on the most basic explicit Euler scheme. We believe that this paper might be helpful for further research on implicit numerical methods, higher-order methods, or even numerical methods for delay differential equations with noisy information under nonstandard assumptions on the right-hand side function.

To do so, we start with an error analysis for the classical explicit Euler scheme under nonstandard assumptions for the right-hand side function $g$ mentioned before, i.e., local H\"older continuity and one-sided Lipschitz condition (Theorem~\ref{est_euler_exact}). Next, we introduce a corruption for the right-hand side function and then we perform an error analysis for the Euler scheme (Theorem~\ref{Euler_error_inexact}). In both cases, we show the dependence of the error on H\"older exponents of the right-hand side function. Furthermore, we establish lower bounds for all algorithms based on exact and inexact information in the certain class of right-hand side functions (Proposition~\ref{prop_Euler_est_ex} and Proposition~\ref{prop_Euler_est_inex} respectively). Consequently, we can provide a result concerning the optimality of the Euler scheme in a particular case (Theorem~\ref{optimal_beta1}). The final contribution is performing numerical experiments to confirm our theoretical results.

The paper is organized as follows. The problem posing with the nonstandard assumptions is given in Section~\ref{sec_prelim}. Section~\ref{sec_euler_exact} on exact information includes a definition of the Euler scheme and a result on the upper bound of the error of this scheme under mentioned nonstandard assumptions. The detailed meaning of the noise, which corrupts values of right-hand side function as well as initial value, is explained in Section~\ref{sec_euler_noisy}, along with the definition of the Euler scheme in case of inexact information and the upper bound of the Euler scheme's error with available only noisy standard information about the right-hand side function. Finally, Section~\ref{sec_optimal} contains the definition of the setting in the spirit of information-based complexity (see \cite{TWW88}), the lower bounds of the Euler scheme, and the optimality of this scheme as the main result of this paper. Section~\ref{sec_num_results} is dedicated to numerical experiments performed using our implementation of the Euler scheme. Last but not least, we put all the auxiliary results (including, a proof of the existence and uniqueness of the solutions of \eqref{ode_gen}) in the Appendix.

\section{Preliminaries}\label{sec_prelim}
For $x,y\in\mathbb{R}^d$ we take $\displaystyle{\langle x,y\rangle=\sum\limits_{k=1}^dx_ky_k}$ and $\|x\|=\langle x,x\rangle^{1/2}$. For $R\in [0,+\infty]$ and $x\in\mathbb{R}^d$, we denote by $B_x(R) = \{ y \in \R^d  \ | \ \|y-x\| \leq R \}$. Of course $B_x(+\infty)=\mathbb{R}^d$. Moreover, for $H\in\mathbb{R}$ we denote by $H^+=\max\{0,H\}$. Additionally, by $U \in \textrm{comp} (\R^d)$ we mean that $U \subset \R^d$ is a compact set.

We assume that the right-hand side function $g:[a,b]\times\mathbb{R}^d\to\mathbb{R}^d$ satisfies the following conditions:
 \begin{itemize}
 		\item [(G1)] $g\in C([a,b]\times\R^d;\R^d)$.
 		\item [(G2)] There exists $K\in (0,+\infty)$ such that for all $(t,y)\in [a,b]\times\R^d$
 		\begin{displaymath}
 		\|g(t,y)\|\leq K(1+\|y\|).
 		\end{displaymath}
 		\item [(G3)] For every $U \in \comp(\mathbb{R}^d)$ there exists $H_U \in \R$ such that for all $t\in [a,b]$, $y_1,y_2\in U$
 	\begin{displaymath}
 		\langle y_1-y_2, g(t,y_1)-g(t,y_2)\rangle \leq H_U \| y_1 - y_2 \|^2.
 	\end{displaymath}
 		\item [(G4)] There exist $\alpha, \beta \in (0,1]$ such that for every $U \in \comp(\mathbb{R}^d)$ there exists $L_U \in (0, \infty)$ such that for all $t_1, t_2 \in [a,b]$, $y_1,y_2\in U$
 	\begin{displaymath}
 		\|g(t_1,y_1)-g(t_2,y_2)\|\leq L_U\Bigl( |t_1-t_2|^\alpha + \|y_1-y_2\|^\beta\Bigr).
 	\end{displaymath}
 	\end{itemize}
The assumption (G2) is the classical global linear growth condition. However, instead of a global Lipschitz condition, we assume that the right-hand side function $g$ locally satisfies one-sided Lipschiz assumption (G3). 
It follows from Lemma \ref{odes_sol_ex1} that under the assumptions (G1)-(G3) the equation \eqref{ode_gen} has a unique solution $z\in C^1([a,b];\mathbb{R}^d)$. The additional assumption (G4), called the $(\alpha,\beta)$-local H\"older condition, is needed to obtain an upper error bound for the Euler scheme that depends on $(\alpha,\beta)$. We stress that under the assumptions (G1)-(G4) the right-hand side function might be highly nonlinear, see, for example, Section~ \ref{sec_num_results}.

In the next section, we investigate the error of the classical Euler scheme under such nonstandard assumptions. Furthermore, we consider the cases where exact and inexact information about the right-hand side function $g$ is available.
\section{Euler scheme under exact information}\label{sec_euler_exact}

Let $n\in \N$, $t_k = a + kh$ for $k=0,1,\dots,n$ where $h=\frac{b-a}{n}$. The Euler scheme is defined as follows
\begin{eqnarray}
\label{ode_euler}
	\left\{ \begin{array}{ll}
	y_0=\xi,\\
	y_{k+1} = y_k + h \cdot g(t_k, y_k) \quad \hbox{for} \ &k = 0,1,\dots, n-1,
	\end{array} \right.
\end{eqnarray}
The  approximation of $z=z(t)$ in  $[a,b]$ is given by a piecewise linear function $l_n:[a,b]\to\mathbb{R}^d$, defined in the following way
\begin{equation}
    l_n(t)=l_{k,n}(t) \ \hbox{for} \ t\in [t_k,t_{k+1}], \ k=0,1,\ldots,n-1,
\end{equation}
where
\begin{equation}
    l_{k,n}(t)=y_k+(t-t_k)\cdot g(t_k,y_k).
\end{equation}
The aim of this section is to prove the following result.
\begin{thm}
\label{est_euler_exact}
Let $\xi \in \R^d$ and let us assume that $g$ satisfies (G1)-(G4). Then there exists $\bar C_1 \in [0,\infty)$ such that for all $n\in\mathbb{N}$
\begin{equation}
    \sup\limits_{a\leq t \leq b}\|z(t)-l_n(t)\|\leq \bar C_1 (1+\|\xi\|)(h^\alpha + h^\beta).
\end{equation}
\end{thm}
\begin{proof}
By the assumption (G2) we have that for all $k=0,1,\ldots,n-1$
\begin{displaymath}
	\|y_{k+1}\|\leq \|y_k\|+h \|g(t_k,y_k)\| \leq (1+hK)\|y_k\|+hK.
\end{displaymath}
Hence, by the discrete version of Gronwall's lemma we get that for all $k=0,1,\ldots,n$ that
\begin{equation}
    \|y_k\|\leq (1+Kh)^k \|\xi\| + (1+Kh)^k -1\leq
    e^{K(b-a)}(1+\|\xi\|),
\end{equation}
and  therefore
\begin{equation}
\label{euler_sol_est_1}
    \max_{n \in \N} \max_{0 \leq k \leq n} \|y_k\| \leq C_1(1+\|\xi\|),
\end{equation}
where $C_1 = C_1(a,b,K) = e^{K(b-a)}$. Note that $C_1\leq C_0$, where $C_0$ is defined in Lemma \ref{odes_sol_ex1}. Now for $k=0,1,\ldots,n-1$ we consider the following local ODE
	\begin{equation}
		\label{def_local_zk}
		z_k'(t)=g(t,z_k(t)), \quad t\in [t_k,t_{k+1}], \quad z_k(t_k)=y_k.
	\end{equation}
	By Lemma~\ref{odes_sol_ex1} there exists a unique solution $z_k:[t_k,t_{k+1}]\to\mathbb{R}^d$ of \eqref{def_local_zk}. From the assumption (G2) and by \eqref{euler_sol_est_1} we get for all $t\in [t_k,t_{k+1}]$, $k=0,1,\ldots,n-1$ that
\begin{equation}
    \|z_k(t)\| 
    \leq e^{K(b-a)}(1+\|\xi\|) + K(b-a) + K \int\limits_{t_k}^t \|z_k(s)\|ds,
\end{equation}
and, by the Gronwall's lemma, 
\begin{equation}
\label{est_zk_local_1}
    \max_{n \in \N} \max_{k=0,1,\dots, n-1} \sup_{t_k \leq t \leq t_{k+1}} \|z_k(t)\| \leq C_2(\|\xi\|+1),
\end{equation}
where $C_2=C_2(a,b,K)=e^{K(b-a)}[e^{K(b-a)} + K(b-a)]$. We have $C_0\leq C_2$. Moreover, by (G2) and \eqref{est_zk_local_1} we have for $t\in [t_k,t_{k+1}]$, $k=0,1,\ldots,n-1$
\begin{equation}
    \|z_k(t) - y_k \|\leq hK\Bigl(1+\sup\limits_{t_k\leq t\leq t_{k+1}}\|z_k(t)\|\Bigr) \leq h C_3(1+\|\xi\|),
\end{equation}
where $C_3 = C_3(a, b, K) = K (1+ C_2) = K (1+ e^{K(b-a)}[e^{K(b-a)} + K(b-a)])$. Therefore, by \eqref{euler_sol_est_1}, \eqref{est_zk_local_1}, \eqref{sol_bound} we have
\begin{equation}
    y_k,z_k(t),z(t)\in B_{0}(C_2(1+\|\xi\|))
\end{equation}
for all $t\in [t_k,t_{k+1}]$, $k=0,1,\ldots,n-1$, $n\in\mathbb{N}$. This and (G4) imply that for \linebreak $U=B_{0}(C_2(1+\|\xi\|))$ there exists $L_1\in (0,+\infty)$ such that for all  $k=0,1,\ldots,n-1$, $n\in\mathbb{N}$
\begin{eqnarray}
\label{error_decomp_0}
    &&\| z_k(t_{k+1}) - y_{k+1} \| \leq
    \int_{t_k}^{t_{k+1}} \| g(s, z_k(s)) - g(t_k, y_k) \| ds \notag \\
    &&
    \leq L_1 \int_{t_k}^{t_{k+1}} ( |s-t_k|^{\alpha} + \|z_k(s)-y_k\|^{\beta}) ds 
    \leq L_1 C_4 (1+\|\xi\|) h (h^{\alpha} + h^{\beta}),
\end{eqnarray}
where $C_4=C_4(a, b, K, \beta) = \max \{1, 2C_3^{\beta}\}$. (We have also used the fact that $(1+x)^\beta\leq 2+x\leq 2(1+x)$ for all $x\geq 0$, $\beta\in (0,1]$.) Furthermore, for $k=0,1,\dots, n-1$ we have
\begin{equation}
\label{error_decomp_1}
    \|z(t_{k+1}) - y_{k+1} \| \leq \|z(t_{k+1}) - z_{k}(t_{k+1})\| + \|z_{k}(t_{k+1}) - y_{k+1} \|,
\end{equation}
and for $t\in [t_k,t_{k+1}]$, $k=0,1,\ldots,n-1$ it  holds that
\begin{equation}
    \frac{d}{dt} \|z(t) - z_k(t)\|^2 = 2 \langle z(t) - z_k(t), g(t, z(t)) - g(t, z_k(t)) \rangle.
\end{equation} 

Therefore, by (G3) applied to $U=B_{0}(C_2(1+\|\xi\|))$ we receive that there exists $H_1 \in \R$ such that
\begin{equation}
    \frac{d}{dt} \|z(t) - z_k(t) \|^2 \leq  2H_1 \|z(t) - z_k(t)\|^2 \notag \\
    \leq 2H_1^+ \|z(t) - z_k(t)\|^2
\end{equation}
for all $t\in [t_k, t_{k+1}]$, $k=0,1,\ldots,n-1$, $n\in\mathbb{N}$. Since $z_k(t_k)=y_k$, we get
\begin{equation}
    \|z(t) - z_k(t) \|^2\leq \|z(t_k) - y_k \|^2+2H^+_1\int\limits_{t_k}^t \|z(s) - z_k(s) \|^2ds,
\end{equation}
for $t\in [t_k,t_{k+1}]$ and,  by the  Gronwall's lemma,
\begin{equation}
\label{error_decomp_2}
    \| z(t) - z_k(t)\| \leq e^{H_1^+ (t-t_k)} \|z(t_k) - y_k\|. 
\end{equation}
Combining \eqref{error_decomp_0}, \eqref{error_decomp_1}, \eqref{error_decomp_2}  we receive for $k=0,1,\ldots,n-1$
\begin{equation}
    \| z(t_{k+1}) - y_{k+1}\| \leq e^{H_1^+ h} \|z(t_k) - y_k \| + L_1 C_4 (1+\|\xi\|) h (h^{\alpha} + h^{\beta}).
\end{equation}
Denoting by $e_k = z(t_k) - y_k$ we arrive at
\begin{displaymath}
    \|e_{k+1}\| \leq e^{H_1^+ h}\|e_k\| + L_1 C_4 (1+\|\xi\|) h (h^{\alpha} + h^{\beta}),
\end{displaymath}
with $e_0 = z(a) - y_0 = 0$. 
If $H_1^+>0$ then by Gronwall's lemma we obtain
\begin{equation}
\label{est_ek_Euler_1}
    \|e_k\| \leq
     \frac{e^{H_1^+ (b-a)}-1}{H_1^+} L_1 C_4 (1+\|\xi\|)(h^\alpha + h^\beta),
\end{equation}
for $k=0,1,\ldots,n$, where $C_4 = C_4(a,b,K,\beta)$. In the particular case, when $H_1^+=0$ we simply have that for $k=0,1,\ldots,n$
\begin{equation}
\label{est_ek_Euler_2}
    \|e_k\|\leq k L_1 C_4 (1+\|\xi\|)(h^\alpha + h^\beta)\leq (b-a)L_1 C_4 (1+\|\xi\|)(h^\alpha + h^\beta).
\end{equation}
For $t\in [t_k,t_{k+1}]$, $k=0,1,\ldots,n$ we have by \eqref{error_decomp_2} that
\begin{eqnarray}
\label{est_linear_Euler_1}
    &&\|z(t)-l_n(t)\|=\|z(t)-l_{k,n}(t)\|\leq \|z(t)-z_k(t)\|+\|z_k(t)-l_{k,n}(t)\|\notag\\
    && \leq e^{H_1^+h}\|e_k\|+\|z_k(t)-l_{k,n}(t)\|,
\end{eqnarray}
where by \eqref{error_decomp_0}
\begin{equation}
\label{est_linear_Euler_2}
    \|z_k(t)-l_{k,n}(t)\|\leq\int\limits_{t_k}^{t_{k+1}} \|g(s,z_k(s))-g(t_k,y_k)\|ds\leq  L_1 C_4 (1+\|\xi\|) h (h^{\alpha} + h^{\beta}).
\end{equation}
Combining \eqref{est_linear_Euler_1}, \eqref{est_linear_Euler_2}, \eqref{est_ek_Euler_1}, \eqref{est_ek_Euler_2} we get the thesis.
\end{proof}

\section{Euler scheme under noisy information}\label{sec_euler_noisy}
In this section we show the upper bound for the error of the Euler scheme in the case of inexact information. Namely, we consider  the corrupted right-hand side functions of the following form
\begin{equation}
    \tilde g(t, y) = g(t,y) + \delta_g(t,y), \quad (t, y) \in [a,b]\times \R^d,
\end{equation}
where $\delta_g: [a,b]\times \R^d \rightarrow \R^d$ is a  {\it corrupting function} that is Borel measurable and  
\begin{equation}
\label{corr_f_assumpt}
    \| \delta_g(t,y) \| \leq \delta (1+ \|y\|)
\end{equation}
for all $(t, y) \in [a,b]\times \R^d$ and $\delta \in [0,1]$. We refer to $\delta$ as to {\it precision parameter}.

In the case of inexact evaluations of $g$, the Euler scheme has the following form
\begin{eqnarray}
\label{ode_euler_noisy}
	\left\{ \begin{array}{ll}
	\tilde y_0=\tilde \xi, \\
	\tilde y_{k+1} = \tilde y_k + h\cdot \tilde g(t_k, \tilde y_k), \quad &k = 0,1,\dots, n-1,
	\end{array} \right.
\end{eqnarray}
where $\tilde \xi\in B_{\xi}(\delta)$. The  approximation of $z=z(t)$ in  $[a,b]$ is given by a piecewise linear function $\tilde l_n:[a,b]\to\mathbb{R}^d$, defined as
\begin{equation}
    \tilde l_n(t)=\tilde l_{k,n}(t) \ \hbox{for} \ t\in [t_k,t_{k+1}], \ k=0,1,\ldots,n-1,
\end{equation}
where
\begin{equation}
    \tilde l_{k,n}(t)=\tilde y_k+(t-t_k)\cdot \tilde g(t_k,\tilde y_k).
\end{equation}
Of course when the information is exact (i.e., $\delta=0$) then $\tilde y_k=y_k$ for $k=0,1,\ldots,n$, and $l_n\equiv\tilde l_n$. Moreover, the Euler algorithm $\tilde l_n$ uses $n$ (noisy) evaluations of $\tilde g$.

In the theorem below we show the upper error bound for the Euler scheme in the case of inexact information about the function $g$. We stress that we have  different bounds in the local Lipschitz case ($\beta=1$) and in the local H\"older case ($\beta\in (0,1)$). This is due to the fact that locally Lipschitz continuous function (i.e. $\beta=1$ in (G4)) satisfies also locally one-sided Lipschitz condition (G3). This is not the case when $g$ is only locally H\"older continuous. We stress that the condition (G3) is crucial for obtaining the upper error bound in the case when $\beta\in (0,1)$, see also Remark \ref{upper_b_euler_inexact}.
\begin{thm}
\label{Euler_error_inexact}
Let $\xi \in \R^d$, let $g$ satisfy (G1)-(G4), and let $\delta_g$ be a Borel measurable corrupting function satisfying \eqref{corr_f_assumpt}. Then there exists $\bar C_2 \in [0,\infty)$ such that for all $n\in\mathbb{N}$, \linebreak $n\geq \lceil 2(b-a)\rceil+1$, $\delta\in [0,1]$, $\tilde \xi\in B_{\xi}(\delta)$  the following holds:
\begin{itemize}
    \item if $\beta=1$ then
    \begin{equation}
    \label{est_ztl_b1}
        \sup\limits_{a\leq t\leq b}\|z(t)-\tilde l_n(t)\|\leq \bar C_2 (1+\|\xi\|)(h^{\alpha}+\delta),
    \end{equation}
\item if $\beta\in (0,1)$ then    
\begin{equation}
\label{est_ztl_b01}
    \sup\limits_{a\leq t\leq b}\|z(t)-\tilde l_n(t)\|\leq \bar C_2 (1+\|\xi\|)\Bigl(h^\alpha + h^\beta+\mathbf{1}_{(0,1]}(\delta)(h^{1/2}+\delta)\Bigr).
\end{equation}
\end{itemize}
\end{thm}
\begin{proof} By (G2) we have for all $k=0,1,\dots, n-1$ that
\begin{displaymath}
     \|\tilde y_{k+1}\| \leq \|\tilde y_k\| + h(K+1) (1+ \|\tilde y_k\|) \notag \\
     \leq (1 + h(K+1)) \|\tilde y_k\| + h(K+1). 
\end{displaymath}
Therefore, by a discrete version of Gronwall's lemma 
we get for $k=0,1,\ldots,n-1$, $n\in\mathbb{N}$, $\delta\in [0,1]$ that
\begin{eqnarray}
\label{est_tilyk}
     &&\|\tilde y_k\| \leq (1+h(K+1))^k \|\tilde y_0\| + (1+h(K+1))^k -1 \notag \\
     &&\leq e^{(b-a)(K+1)} (\|\tilde \xi\| +1)\leq K_2 (1+\|\xi\|)
\end{eqnarray}
where $K_2=K_2(a,b,K)= 2e^{(b-a)(K+1)}$. Since $K_2\geq C_1$, by \eqref{euler_sol_est_1},\eqref{est_tilyk}  we have
\begin{equation}
\label{ball_2}
    y_k,\tilde y_k\in B_0(K_2(1+\|\xi\|))
\end{equation}
for all $k=0,1,\ldots,n-1$, $n\in\mathbb{N}$, $\delta\in [0,1]$.
Let us denote $e_k = \tilde y_k - y_k$ for $k = 0,1,\dots,n$. Then we have that  $\|e_0\| = \|\tilde y_0 - y_0\| = \|\tilde \xi - \xi\| \leq \delta$  and for $k = 0,1,\dots, n-1$
\begin{equation}
\label{ek_decomp_1}
     e_{k+1} =e_k + h \left[ g(t_k,\tilde y_k) - g(t_k, y_k)\right] + h \delta_g(t_k,\tilde y_k).
\end{equation}
We consider two cases, when $\beta=1$ and $\beta\in (0,1)$.

If $\underline{\beta=1}$ then (G4) implies that for $U=B_0(K_2(1+\|\xi\|))$ there exist $L_2\in (0,+\infty)$  such that for $k=0,1,\ldots,n-1$, $n\in\mathbb{N}$, $\delta\in [0,1]$
\begin{equation}
    \|e_{k+1}\|\leq (1+hL_2)\|e_k\|+h\delta(1+\|\xi\|)(1+K_2).
\end{equation}
By the discrete version of the Gronwall's lemma we get for all $n\in\mathbb{N}$, $\delta\in [0,1]$
\begin{equation}
\label{yktyk_est_bet_1}
    \max\limits_{0\leq k\leq n}\|y_k-\tilde y_k\|\leq (1+\|\xi\|)\Bigl(e^{L_2(b-a)}+\frac{e^{L_2(b-a)}-1}{L_2}(1+K_2)\Bigr)\delta.
\end{equation}
Hence, by Theorem \ref{est_euler_exact} we get that there exists $\tilde C_1\in [0,+\infty)$ such that for all $t\in [t_k,t_{k+1}]$, $k=0,1,\ldots,n$, $n\in\mathbb{N}$, $\delta\in [0,1]$
\begin{equation}
\label{est_ztl_beta_1}
    \|z(t)-\tilde l_n(t)\|\leq \tilde C_1 (1+\|\xi\|) h^\alpha+\|\tilde l_{k,n}(t)-l_{k,n}(t)\|,
\end{equation}
and, by \eqref{yktyk_est_bet_1},
\begin{eqnarray}
&&\|\tilde l_{k,n}(t)-l_{k,n}(t)\|\leq (1+hL_2)\|y_k-\tilde y_k\|+h\delta(1+\|\tilde y_k\|)\notag\\
\label{beta_1}
&&\leq \tilde C_1 (1+\|\xi\|)\delta+h\delta (1+\|\xi\|)(1+K_2).
\end{eqnarray}
From \eqref{est_ztl_beta_1}, \eqref{beta_1} we get \eqref{est_ztl_b1}.

If $\underline{\beta\in (0,1)}$ then we have to proceed in a different way, see also Remark \ref{upper_b_euler_inexact}. From \eqref{ek_decomp_1} we have for $k=0,1,\ldots,n-1$ that 
\begin{equation}
\label{rownosc_kwadratow_ode_noisy}
     \|e_{k+1} - h \delta_g(t_k,\tilde y_k)\|^2 = \|e_k + h \left[ g(t_k,\tilde y_k) - g(t_k, y_k)\right] \|^2.
\end{equation}
We have
\begin{eqnarray}
     &&\|e_{k+1} - h \delta_g(t_k,\tilde y_k)\|^2 = \|e_{k+1}\|^2 - 2h \langle e_{k+1}, \delta_g(t_k,\tilde y_k) \rangle + h^2 \|\delta_g(t_k,\tilde y_k)\|^2 \notag \\
     &&\geq \|e_{k+1}\|^2 - 2h \langle e_{k+1}, \delta_g(t_k,\tilde y_k) \rangle. \label{rownosc_kwadratow_L}
\end{eqnarray}
On the other hand, \eqref{ball_2} and (G3), (G4) imply that for $U=B_0(K_2(1+\|\xi\|))$ there exist $L_2\in (0,+\infty)$, $H_2\in\mathbb{R}$ such that for $k=0,1,\ldots,n-1$, $n\in\mathbb{N}$, $\delta\in [0,1]$
\begin{eqnarray}
     &&\|e_k + h \left[ g(t_k,\tilde y_k) - g(t_k, y_k)\right] \|^2\notag\\
     &&= \|e_k\|^2 + 2h \left\langle e_k,  g(t_k,\tilde y_k) - g(t_k, y_k) \right\rangle + h^2 \| g(t_k,\tilde y_k) - g(t_k, y_k)\|^2 \notag \\
     &&\leq \|e_k\|^2 + 2h H_2^+ \|e_k\|^2 + h^2 L_2^2 \|e_k\|^{2\beta} 
     \leq (1 + 2h H_2^+) \|e_k\|^2 + h^2 L_2^2 (1+ \|e_k\|^2) \notag \\
     &&\leq (1 + K_3 h) \|e_k\|^2 + h^2 L^2, \label{rownosc_kwadratow_R}
\end{eqnarray}
where $K_3=(2H_2^++L_2^2)$. (We have also used the fact that $x^{2\beta}=(x^2)^{\beta}\leq 1+x^2$ for all $x\in\mathbb{R}$, $\beta\in (0,1]$.) Combining \eqref{rownosc_kwadratow_ode_noisy}, \eqref{rownosc_kwadratow_L} and \eqref{rownosc_kwadratow_R} we get
\begin{equation}
\label{error_inequality}
     \|e_{k+1}\|^2 \leq (1 + K_3 h) \|e_k\|^2 + h^2 L_2^2 + 2h \langle e_{k+1}, \delta_g(t_k,\tilde y_k) \rangle.
\end{equation}
Moreover,
\begin{equation}
     \langle e_{k+1}, \delta_g(t_k,\tilde y_k) \rangle 
     \leq \|e_{k+1}\| \cdot \|\delta_g(t_k,\tilde y_k) \| \leq\frac{1}{2}\|e_{k+1}\|^2+\delta^2(1+K_2^2)(1+\|\xi\|)^2.
     \label{error_inequality_dot_prod}
\end{equation}
Combining \eqref{error_inequality} with \eqref{error_inequality_dot_prod} we end up with
\begin{equation}
     \|e_{k+1}\|^2 \leq (1 + K_3 h) \|e_k\|^2 + h^2 L_2^2 + h  \|e_{k+1}\|^2 + 2h\delta^2 (1+K_2^2)(1+\|\xi\|)^2.
\end{equation}
Since $h \in (0,\frac{1}{2})$, we have  $0 < \frac{1}{1-h} \leq 1+2h \leq 2$ and for all $k = 0,1,\dots, n-1$
\begin{eqnarray}
     &&\|e_{k+1}\|^2 \leq  (1+2h)(1 + K_3 h) \|e_k\|^2 + K_4 (1+\|\xi\|)^2 h(h+ \delta^2)\notag\\
     &&\leq (1+K_5h)\|e_k\|^2 + K_4 (1+\|\xi\|)^2 h(h+ \delta^2),
\end{eqnarray}
 where $K_4=2\max\{L_2^2,1+K_2^2\}$, $K_5=2+3K_3>0$. By the discrete version of the Gronwall's lemma we obtain for all $k = 0,1,\dots, n$
\begin{eqnarray}
     &&\|e_k\|^2 \leq (1 + K_5 h)^k \|e_0\|^2 + \frac{(1 + K_5 h)^k - 1}{K_5} K_4 (1+\|\xi\|)^2 (h+ \delta^2) \notag \\
     &&\leq e^{K_5 (b-a)} \delta^2 + \frac{e^{K_5(b-a)} - 1}{K_5} K_4 (1+\|\xi\|)^2 (h+ \delta^2) \notag \\
     &&\leq K_6(1+\|\xi\|)^2 (h+\delta^2),
\end{eqnarray}
where $K_6=2\max\{e^{K_5(b-a)},\frac{e^{K_5(b-a)}-1}{K_5}K_4\}$. Therefore, for all  $n\in\mathbb{N}$, $\delta\in [0,1]$
\begin{equation}
\label{est_yktyk}
    \max_{k = 0,1,\dots,n} \|y_k - \tilde y_k\| \leq K_7(1+\|\xi\|)(h^{\frac{1}{2}}+\delta),
\end{equation}
with $K_7=K_6^{1/2}$. Again, by Theorem \ref{est_euler_exact} we get that there exists $\tilde C_2,\tilde C_3\in [0,+\infty)$ such that for all $t\in [t_k,t_{k+1}]$, $k=0,1,\ldots,n$, $n\in\mathbb{N}$, $\delta\in [0,1]$
\begin{equation}
\label{est_ztl_beta_11}
    \|z(t)-\tilde l_n(t)\|\leq \bar C_1 (1+\|\xi\|) (h^\alpha+h^{\beta})+\|\tilde l_{k,n}(t)-l_{k,n}(t)\|,
\end{equation}
and, by \eqref{est_yktyk},
\begin{eqnarray}
&&\|\tilde l_{k,n}(t)-l_{k,n}(t)\|\leq \|y_k-\tilde y_k\|+hL_2\|y_k-\tilde y_k\|^{\beta}+h\delta(1+\|\tilde y_k\|)\notag\\
\label{beta_11}
&&\leq \tilde C_2 (1+\|\xi\|)(h^{1/2}+\delta+h^{1+\frac{\beta}{2}}+h\delta^{\beta})+h\delta (1+\|\xi\|)(1+K_2)\notag\\
&&\leq \tilde C_3(1+\|\xi\|)(h^{1/2}+\delta).
\end{eqnarray}
Finally, from \eqref{est_ztl_beta_11}, \eqref{beta_11} we get \eqref{est_ztl_b01}.
\end{proof}
\begin{rem}
\label{upper_b_euler_inexact}
\rm
 Note that if $\beta\in (0,1)$ then we can only get that
 \begin{equation}
     \|e_{k+1}\|\leq \|e_k\|+hL_2\|e_k\|^{\beta}+h\delta(1+\|\xi\|)(1+K_2)
 \end{equation}
 for $k=0,1,\ldots,n-1$. Since $\|e_k\|^{\beta}\leq 1+\|e_k\|$, it holds
 \begin{equation}
     \|e_{k+1}\|\leq (1+hL_2)\|e_k\| + hL_2+h\delta(1+\|\xi\|)(1+K_2),
 \end{equation}
 which in turns implies that
 \begin{equation}
     \max\limits_{0\leq k\leq n}\|e_k\|\leq e^{L_2(b-a)}\delta+\frac{e^{L_2(b-a)}-1}{L_2}(L_2+\delta)(1+\|\xi\|)(1+K_2)=O(1),
 \end{equation}
 as $\delta\to 0^+$.  This estimate is obviously too rough.
\end{rem}
\section{Lower bounds and optimality of the Euler scheme in the worst-case setting}\label{sec_optimal}

In this section we provide optimality analysis of the Euler algorithm in the case of inexact information about the right-hand side function. We settle our problem in the worst-case model of error (see \cite{TWW88}), which, in particular, allows us to investigate lower error bounds.

Let $\alpha, \beta \in (0,1]$, $K,L\in (0,+\infty)$, $H\in\mathbb{R}$, $R\in [0,+\infty]$. We consider the class $F_R^{\alpha, \beta}=F^{\alpha,\beta}_R(a,b,d,K,H,L)$ of pairs $(\xi, g)$ that satisfy the following conditions:
\begin{itemize}
    \item [(A0)] $\|\xi\| \leq K$,
	\item [(A1)] $g\in C([a,b]\times\R^d;\R^d)$,
	\item [(A2)] $\|g(t,y)\|\leq K(1+\|y\|)$ for all $(t,y)\in [a,b]\times\R^d$,
	\item [(A3)] $\langle y_1-y_2, g(t,y_1)-g(t,y_2)\rangle \leq H \| y_1 - y_2 \|^2$ for all $t\in [a,b]$, $y_1,y_2\in B_0(R)$,
	\item [(A4)] $	\|g(t_1,y_1)-g(t_2,y_2)\|\leq L\Bigl( |t_1-t_2|^\alpha + \|y_1-y_2\|^\beta\Bigr)$ for all $t_1, t_2 \in [a,b]$, $y_1,y_2\in B_0(R)$.
\end{itemize}
We refer to  $\alpha, \beta \in (0,1]$, $-\infty<a<b<+\infty$, $d \in \N$, $K, L \in [0, \infty)$, $H \in \R$, $R \in [0, +\infty]$ as to the {\it parameters of the class}. Except for $a,b,d$ the parameters are not known and will not be used by the defined algorithms. 

The class $F^{\alpha,\beta}_{\infty}$ consists of functions $g$ that satisfy conditions (A3), (A4) globally on the whole domain $[a,b]\times\mathbb{R}^d$, while functions from $F^{\alpha,\beta}_0$ might not satisfy the one-sided Lipschitz condition and the $(\alpha,\beta)$-H\"older condition even locally.
We have that $F_{\infty}^{\alpha,\beta}\subset F_{R_1}^{\alpha,\beta}\subset F_{R_2}^{\alpha,\beta}\subset F_{0}^{\alpha,\beta}$ for all $0\leq R_2\leq R_1\leq +\infty$.

We assume that an algorithm may use exact or noisy evaluations of the function $g$. Consequently, we can have access to the function $g$ through its possibly perturbed values $\tilde g$ at given points $(t,y)$. To be more precise we consider the following set of corrupting functions
\begin{eqnarray}
    &&\mathcal{K}(\delta)=\{\tilde\delta:[a,b]\times\mathbb{R}^d\to\mathbb{R}^d \ : \  \tilde\delta-\hbox{Borel measurable}, \ \|\tilde\delta(t,y)\|\leq \delta(1+\|y\|)\notag\\
    &&\quad\quad\quad\quad\quad \hbox{for all} \ t\in [a,b], y\in\mathbb{R}^d\}.
\end{eqnarray}
Of course $\mathcal{K}(\delta_1)\subset\mathcal{K}(\delta_2)$ for all $0\leq \delta_1\leq \delta_2\leq 1$. We also set
\begin{equation}
    V_{(\xi,g)}(\delta)=B_{\xi}(\delta)\times V_g(\delta),
\end{equation}
where
\begin{equation}
    V_g(\delta)=\bigcup_{\tilde\delta\in\mathcal{K}(\delta)}\{g+\tilde\delta\}.
\end{equation}
For $0\leq\delta_1\leq\delta_2\leq 1$ we have $V_{(\xi,g)}(\delta_1)\subset V_{(\xi,g)}(\delta_2)$ and $V_{(\xi,g)}(0)=\{(\xi,g)\}$.

Let $(\xi, g)\in F_0^{\alpha, \beta}$ and $(\tilde\xi,\tilde g)\in V_{(\xi,g)}(\delta)$. 
A vector of noisy information about $(\xi, g)$ is of the following form
\begin{displaymath}
\mathcal N(\tilde \xi, \tilde g) = [ \tilde g(t_0, y_0), \tilde g(t_1, y_1), \dots, \tilde g(t_{i-1}, y_{i-1}), \tilde \xi ]
\end{displaymath}
where $i\in\mathbb{N}$ is a total number of noisy evaluations of $g$. The evaluation points $(t_k,y_k)$ can be chosen in adaptive way. This means that there exist  functions $\psi_j$ such that
\begin{equation}
    (t_0,y_0)=\psi_0(\tilde\xi),
\end{equation}
and for $j=1,2,\ldots,i-1$
\begin{equation}
    (t_j,y_j)=\psi_j\Bigl(\tilde g(t_0,y_0),\ldots, \tilde g(t_{j-1},y_{j-1}),\tilde\xi\Bigr).
\end{equation}
An algorithm that approximates $z$ and uses the vector of (noisy) information $\mathcal{N}$ is defined as 
\begin{equation}
\label{def_alg}
    \mathcal{A}(\tilde\xi,\tilde g)=\varphi(\mathcal N(\tilde \xi, \tilde g)),
\end{equation}
for some function
\begin{equation}
    \varphi:\mathbb{R}^{(i+1)d}\to B([a,b];\mathbb{R}^d),
\end{equation}
where $B([a,b];\mathbb{R}^d)$ is the linear space of all functions $l:[a,b]\to\mathbb{R}^d$ bounded in the supremum norm. For a given $n\in\mathbb{N}$ we define $\Phi_n$ as a class of all algorithms of the form \eqref{def_alg} for which the total number of  evaluations $i$ is at most $n$. For example, $\tilde l_n\in\Phi_n$.

For a fixed $(\xi,g)\in F^{\alpha,\beta}_0$ the error of $\mathcal{A}\in\Phi_n$ is defined as follows
\begin{equation}
    e(\mathcal{A},\xi,g,\delta)=\sup\limits_{(\tilde\xi,\tilde g)\in V_{(\xi,g)}(\delta)}\sup\limits_{a\leq t\leq b}\|z(\xi,g)(t)-\mathcal{A}(\tilde\xi,\tilde g)(t)\|,
\end{equation}
while the worst-case error of the algorithm $\mathcal{A}$ in a subclass  $\mathcal{G}\subset F^{\alpha,\beta}_0$
\begin{equation}
    e(\mathcal{A},\mathcal{G},\delta)=\sup\limits_{(\xi,g)\in \mathcal{G}}e(\mathcal{A},\xi,g,\delta).
\end{equation}
We investigate behavior of the $n$th minimal error defined as below
\begin{equation}
    e_n(\mathcal{G},\delta)=\inf_{\mathcal{A}\in\Phi_n}e(\mathcal{A},\mathcal{G},\delta).
\end{equation}

The following worst-case results in the class $F^{\alpha,\beta}_R$ (with  suitable $R$'s)  for the Euler algorithm $\tilde l_n$, follow from the proofs of Theorems \ref{est_euler_exact}, \ref{Euler_error_inexact}. Note that the error bounds in the Theorems \ref{est_euler_exact}, \ref{Euler_error_inexact} are shown for a fixed $(\xi,g)$, while now we are considering the worst-case error in the class $F^{\alpha,\beta}_R$. For the clarity we present  upper bounds on the error of the Euler algorithm in the case of exact and inexact information separately.
\begin{prop} (The case of exact information)
\label{prop_Euler_est_ex}
    Let $R_1=e^{K(b-a)}(1+K) \left( e^{K(b-a)} + K(b-a) \right)$. Then there exists $C\in (0,+\infty)$, depending only on the parameters of the class $F^{\alpha,\beta}_{R_1}$, such  for all $(\xi, g) \in F_{R_1}^{\alpha, \beta}$, $n \in \mathbb{N}$ it holds
\begin{equation}\label{main_prep}
    \sup\limits_{a\leq t\leq b} \|z(t) -l_n(t) \| \leq C(h^\alpha + h^\beta).
\end{equation} 
\end{prop}
\begin{prop} (The case of inexact information)
\label{prop_Euler_est_inex}
    Let $R_2=\max\{R_1,K_2(1+K)\}=e^{K(b-a)}(1+K)\max\{e^{K(b-a)}+K(b-a),2\}$. Then there exists $C\in (0,+\infty)$, depending only on the parameters of the class $F^{\alpha,\beta}_{R_2}$, such that for all $n\geq \lceil 2(b-a)\rceil+1$, $\delta\in [0,1]$, $(\xi,g)\in F^{\alpha,\beta}_{R_2}$, $(\tilde\xi,\tilde g)\in V_{(\xi,g)}(\delta)$ the following holds
    \begin{itemize}
        \item if $\beta=1$ then
        \begin{equation}
        \label{beta_1_prep}
    \sup\limits_{a\leq t\leq b} \|z(t) -\tilde l_n(t) \| \leq  C(h^{\alpha}+\delta).
\end{equation}
        \item if $\beta\in (0,1)$ then
\begin{equation}
\label{beta_01}
    \sup\limits_{a\leq t\leq b} \|z(t) -\tilde l_n(t) \| \leq  C\Bigl(h^{\alpha}+h^{\beta}+\mathbf{1}_{(0,1]}(\delta)\cdot (h^{1/2}+\delta)\Bigr).
\end{equation} 
\end{itemize}
\end{prop}
In the special case when $\beta=1$ we get the following sharp bounds on the $n$th minimal error in the class $F^{\alpha,1}_{R_2}$.
\begin{thm}
\label{optimal_beta1}
Let $\alpha\in (0,1]$, $\beta=1$ and $R_2$ defined as in Proposition \ref{prop_Euler_est_inex}. Then it holds
    \begin{equation}
    \label{sharp_en}
        e_n(F^{\alpha,1}_{R_2},\delta)=\Theta(n^{-\alpha}+\delta),
    \end{equation}
    as $n\to +\infty$, $\delta\to 0^+$. The  Euler algorithm $\tilde l_n$ is  the optimal one.
\end{thm}
\begin{proof}
The upper bound follows from \eqref{beta_1_prep} in Proposition \ref{prop_Euler_est_inex} and the fact that $\tilde l_n\in\Phi_n$.

Now we turn to the lower bound. We consider the following class $$\mathcal{G}^{\alpha,1}_{R_2} = \{(0, g) \in F^{\alpha,1}_{R_2} \ | \  g(t,y) = g(t,0) \textrm{ for all } (t,y) \in [a,b]\times\R^d \},$$
which is a subclass of $F^{\alpha,1}_{R_2}$.
Moreover, for $(\xi,g)\in\mathcal{G}^{\alpha}_{R_2}$ we have that $\displaystyle{z(\xi,g)(t)=\int\limits_a^t g(s,0)ds}$, $t\in [a,b]$. 
Hence, 
\begin{equation}
\label{low_b_1}
    e_n(F^{\alpha,1}_{R_2},\delta)\geq  e_n(\mathcal{G}^{\alpha,1}_{R_2},0)\geq\inf_{\mathcal{A}\in\Phi_n}\sup\limits_{(0,g)\in\mathcal{G}^{\alpha,1}_{R_2}}\|z(0,g)(b)-\mathcal{A}(0,g)(b)\|=\Omega(n^{-\alpha}),
\end{equation}
which follows from the result on lower bound for a Riemann
integration problem of H\"older continuous functions under exact information, see Proposition 1.3.9, page 34 in \cite{NOV}. Moreover, using analogous argumentation as in the proof of Theorem 3 in \cite{TBPP2022} we get, in the case of inexact information, that
\begin{equation}
\label{low_b_2}
    e_n(F^{\alpha,1}_{R_2},\delta)\geq (b-a)\delta.
\end{equation}
From \eqref{low_b_1} and \eqref{low_b_2} we get the lower bound in \eqref{sharp_en}.
\end{proof}
\begin{rem}
\rm
    By considering the autonomous problem $z'(t) = g(z(t))$, $z(a)=\xi$, $t\in [a,b]$, with globally Lipschitz continuous function $g$ we know, from the results of \cite{Kac2}, that we can only get the lower bound $\Omega(h)$. Hence, under our assumptions there is a gap between the upper bound $O(h^{\beta})$ and the lower bound $\Omega(h)$ - we conjecture that the optimal error bound is $\Theta (h^{\beta})$.
\end{rem}
\section{Numerical results}\label{sec_num_results}

We have conducted various numerical tests to further prove \eqref{main_prep}. The exact solutions for the analyzed problems are unknown, so we proceeded as follows to estimate the empirical rate of convergence. In parallel, we computed the approximated solution and a numerical reference solution with a smaller step size $h$ (1000 times smaller than in an approximated solution). Then we computed an error as a difference between those two solutions.

The theoretical convergence rate is known. We presented the test results in the log-log plots to test a numerical convergence rate. On the y axis there is $-\log_{10} \textrm{(err)}$ and on the x axis the $\log_{10}n$, where $\textrm{err}$ and $n$ are given. Then, the linear regression slope corresponds to the empirical convergence rate and is presented for the case with exact information.

We computed only one approximated solution to get the numerical convergence rate for the case without noise. To gain the numerical convergence rate for a noisy case, we conducted 200 Monte Carlo tries, and then we took the biggest value of the error to present in the plot to mimic the worst-case setting behavior.

\subsection{Additive example}
For all $(t,x)\in [a,b]\times\mathbb{R}$ let us consider the function
\begin{equation}
\label{exampl_g1}
    g(t,x)=A|t|^{\alpha}x\sin(x^2+1)+B_1\sgn(x) |x|^{\varrho_1}+B_2\sgn(x) |x|^{\varrho_2},
\end{equation}
where $A\in \mathbb{R}$, $B_1,B_2\in (-\infty,0]$, $\alpha,\varrho_1,\varrho_2\in (0,1]$. Note that the function \eqref{exampl_g1} satisfies assumptions (G1)-(G4), see Fact \ref{ex_g_g1g4}.

In the first experiment, we take the following values of the parameters $a=0$, $b=10$, $A=2$, $B_1=-2$, $B_2=-0.5$, $\alpha=1$ and $\varrho_1=1$. We manipulate with $\varrho_2$ and noise level $\delta$.

In Figure~\ref{pic:sol_alpha_1} we have several reference solutions for $\varrho_2 \in \{1.0, 0.75, 0.5, 0.25 \}$ and $\delta=0.0$. In Figure~\ref{pic:err_rho2} we have test results for $\delta=0$ (exact information) and $\delta \in \{ 0.1, 0.01, 0.001 \}$ (noisy information).

\pic{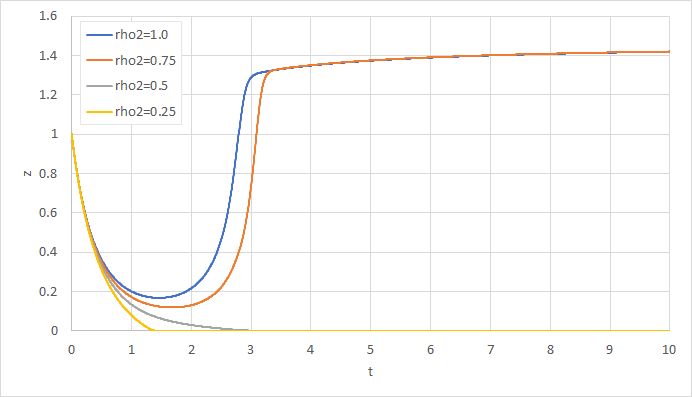}{0.8}{An approximated solution computed on the dense mesh, treated as the reference solution of the model \eqref{exampl_g1} with: $\varrho_2=1$, $\varrho_2=0.75$, $\varrho_2=0.5$, and $\varrho_2=0.25$.}{pic:sol_alpha_1}

\begin{figure}[!htbp]
	\centering
	\begin{subfigure}[t]{0.45\textwidth}
		\centering
		\includegraphics[height=3.5cm]{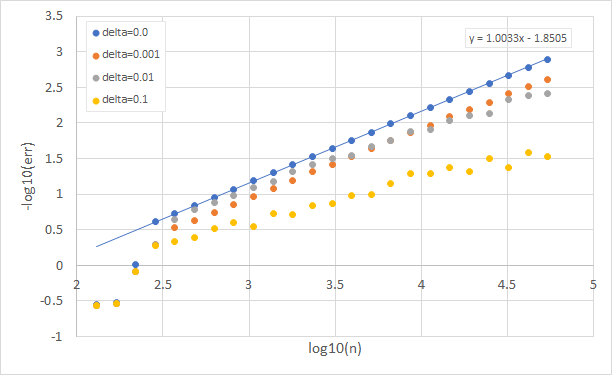}
		\caption{$-\log_{10} \textrm{(err)}$ vs. $\log_{10}n$ for $\varrho_2=1$.}\label{pic:err_rho2_1.0}		
	\end{subfigure}
	\quad \quad
	\begin{subfigure}[t]{0.45\textwidth}
		\centering
		\includegraphics[height=3.5cm]{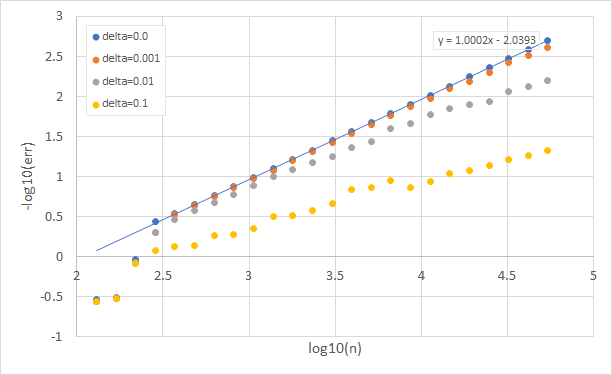}
		\caption{$-\log_{10} \textrm{(err)}$ vs. $\log_{10}n$ for $\varrho_2=0.75$.}\label{pic:err_rho2_0.75}
	\end{subfigure}		
	\begin{subfigure}[t]{0.45\textwidth}
		\centering
		\includegraphics[height=3.5cm]{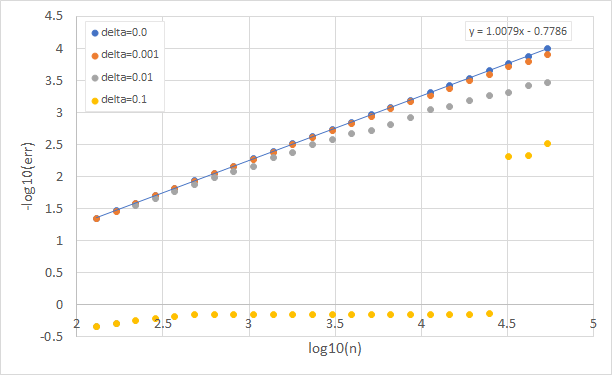}
		\caption{$-\log_{10} \textrm{(err)}$ vs. $\log_{10}n$ for $\varrho_2=0.5$.}\label{pic:err_rho2_0.5}	
	\end{subfigure}
	\quad
	\begin{subfigure}[t]{0.45\textwidth}
		\centering
		\includegraphics[height=3.5cm]{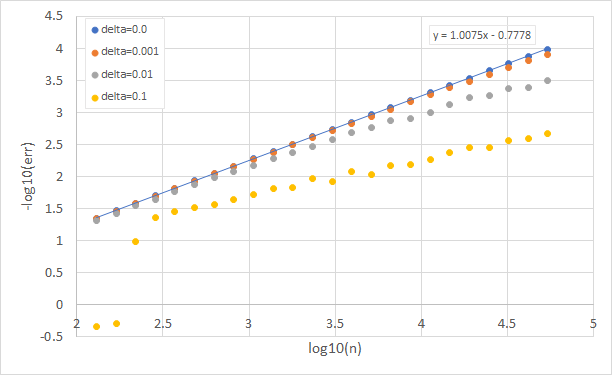}
		\caption{$-\log_{10} \textrm{(err)}$ vs. $\log_{10}n$ for $\varrho_2=0.25$. }\label{pic:err_rho2_0.25}
	\end{subfigure}
	\caption{$-\log_{10} \textrm{(err)}$ vs. $\log_{10}n$ for the model \eqref{exampl_g1}.}\label{pic:err_rho2}
\end{figure}
%

In the next experiment we take the following values of the parameters: $a=0$, $b=10$, $A=2$, $B_1=-2$, $B_2=-0.5$, $\varrho_1=1$, $\varrho_2=0.75$ but $\alpha=0.4$. In Figure~\ref{pic:sol_rho_0.75} we compare reference solutions for $\alpha \in \{ 0.4, 1.0 \}$ with $\delta=0$. In Figure~\ref{pic:err_alpha} we have test results for $\delta=0$ and $\delta \in \{ 0.1, 0.01, 0.001 \}$.

From Proposition~\ref{prop_Euler_est_ex} and  Fact~\ref{ex_g_g1g4} we have the theoretical convergence rate is \linebreak $O\Bigl(h^{\alpha}+h^{\min \{\varrho_1, \varrho_2 \} }+\mathbf{1}_{(0,1]}(\delta)\cdot (h^{1/2}+\delta)\Bigr)$.
As we see in Figure~\ref{pic:err_rho2} and Figure~\ref{pic:err_alpha} the bigger $\delta$, the less steep are the empirical convergence rates.

\pic{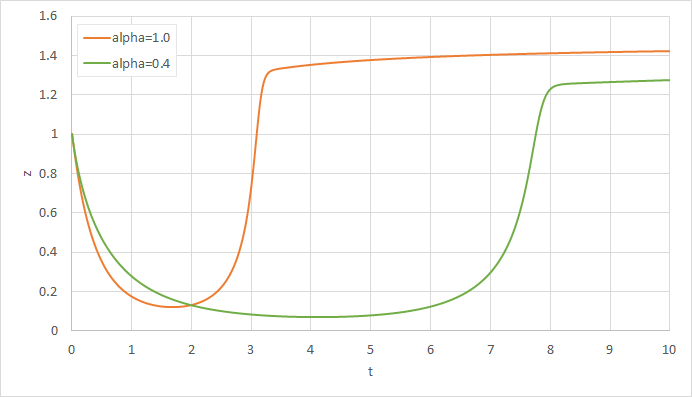}{0.8}{An approximated solution computed on the dense mesh, treated as the reference solution of the model \eqref{exampl_g1} with: $\alpha=0.4$, $\alpha=1.0$.}{pic:sol_rho_0.75}

\pic{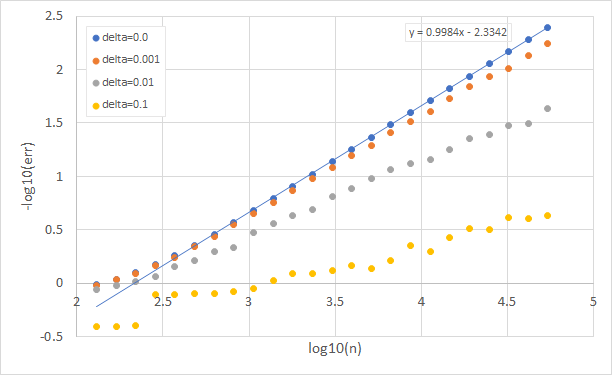}{0.8}{$-\log_{10} \textrm{(err)}$ vs. $\log_{10}n$ for the model \eqref{exampl_g1} with $\alpha=0.4$.}{pic:err_alpha}

\subsection{Multiplicative example}

Let us consider a family of the following functions
\begin{equation}
\label{exampl_g2}
    g(t,x)= \gamma(t) h(x) f(x), \quad (t,x)\in [a,b]\times\mathbb{R}
\end{equation}
where $h(x) = - \sgn(x) |x|^\varrho$, $\varrho\in (0,1]$, and for $\gamma:[a,b]\to\mathbb{R}$, $f:\mathbb{R}\to\mathbb{R}$ the following conditions are satisfied:
\begin{itemize}
    \item there exists $L_\gamma \geq 0$, $\alpha \in (0,1]$ such that for all $t, s \in [a,b]$
\begin{equation}
\label{example_g2_gamma1}
    |\gamma(t) - \gamma (s)| \leq L_\gamma |t-s|^\alpha,
\end{equation}
    \item for all $t \in [a, b]$
    \begin{equation}
    \label{example_g2_gamma2}
        \gamma(t)\geq 0,
    \end{equation}
    \item there exists $D\geq 0$ such that for all $x \in \R$
    \begin{equation}
    \label{example_g2_f1}
        0 \leq f(x) \leq D,
    \end{equation}
    \item there exists $L_f \geq 0$ such that for all $x, y \in \R$
    \begin{equation}
    \label{example_g2_f2}
        |f(x) - f(y)| \leq L_f |x-y|.
    \end{equation}
\end{itemize}
It turns out that the function \eqref{exampl_g2} satisfies the assumptions (G1)-(G4), see Fact \ref{ex_multipl_g1g4}.

For the testing purpose we take the following particular function
\begin{equation}
    g(t,x)=-0.3\cdot |\sin(\pi \cdot t+1)|^{2/3}\cdot \sgn(x) |x|^{2/3}\cdot |\cos(x)|,
\end{equation}
with  $\varrho = \frac{2}{3}$, $a=-3.5$, $b=3.5$, and with the initial-condition $\xi = 3$.
%
%

In Figure~\ref{pic:sol_muliplicative} we show a reference solution for exact information ($\delta=0$). In Figure~\ref{pic:err_multiplicative} we present the test results for $\delta=0$ and $\delta \in \{ 0.1, 0.01, 0.001 \}$.

The reference solution in Figure~\ref{pic:sol_muliplicative} is a little bit bumpy, therefore we can expect that the bigger noise, the more bumpy is the solution and the more approximation steps must be made to obtain a satisfactory approximated solution, see Figure~\ref{pic:err_multiplicative}.

\pic{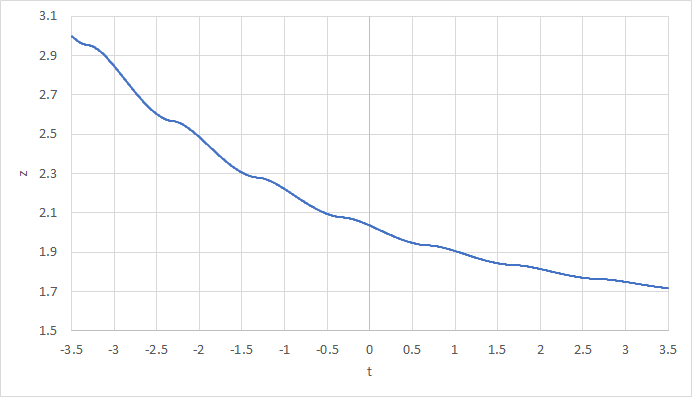}{0.8}{A reference solution of the model \eqref{exampl_g2}.}{pic:sol_muliplicative}

\pic{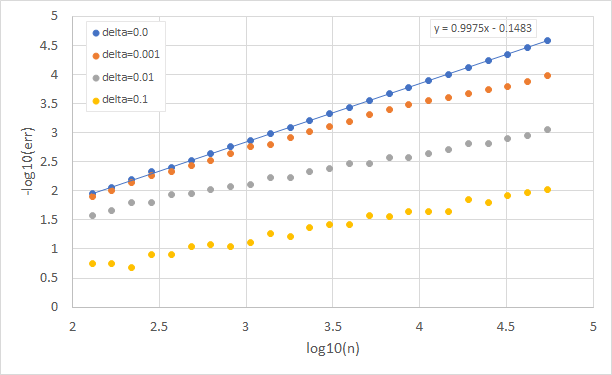}{0.8}{$-\log_{10} \textrm{(err)}$ vs. $\log_{10}n$ for the model \eqref{exampl_g2}.}{pic:err_multiplicative}

{\bf Acknowledgments}
We would like to thank two anonymous reviewers for valuable comments and suggestions that allowed us to improve the
quality of the paper.

\section{Appendix}\label{sec_appendix}
In this section we give some auxiliary facts about solution $z$ of the ODE \eqref{ode_gen} and properties of the right-hand side function \eqref{exampl_g1}. We use  the following version of Peano's theorem. 

\begin{thm}[Peano's theorem; \cite{gorn_ing}, page 292., Theorem 70.4] \label{peano}
If the function $g$ satisfies the assumptions (G1) and (G2), then the initial value problem \eqref{ode_gen} has at least one solution that belongs to $C^1([a,b]\times\mathbb{R}^d;\mathbb{R}^d)$.
\end{thm}
\begin{lem}
\label{odes_sol_ex1}
Let $\xi\in\mathbb{R}^d$ and let us assume that $g$ satisfies the assumptions (G1), (G2) and (G3). Then there exists a unique solution $z=z(\xi,g)\in C^1([a,b];\mathbb{R}^d)$ of \eqref{ode_gen} and 
\begin{equation}
\label{sol_bound}
    \sup_{a \leq t \leq b} \| z(t) \| \leq K_0,
\end{equation}
where $K_0 = K_0(a,b,K,\xi) = (\|\xi\| + K(b-a)) e^{K(b-a)}\leq C_0(1+\|\xi\|)$ and $C_0=C_0(a,b,K)=\max\{1,K(b-a)\}e^{K(b-a)}$.
\end{lem}
\begin{proof}
By (G1), (G2) and from Theorem \ref{peano} the equation \eqref{ode_gen} has at least one solution $z=z(\xi,g) \in C^1([a,b]; \R^d)$. Hence, by (G2) we get for all $t\in [a,b]$
\begin{equation}
    \|z(t)\|\leq \|\xi\|+K(b-a)+K\int\limits_a^t\|z(s)\|ds,
\end{equation}
and from the  Gronwall's lemma we obtain that
\begin{equation}
\sup_{a \leq t \leq b} \| z(t) \| \leq K_0,
\end{equation}
where $K_0 = K_0(a,b,K,\xi) = (\|\xi\| + K(b-a)) e^{K(b-a)}$. Hence, $z(t)\in B_0(K_0)$ for all $t\in [a,b]$.

Let us consider two solutions $z=z(\xi,g)$ and $\tilde z=\tilde z(\xi,g)$  of \eqref{ode_gen}. We have that $z(t),\tilde z(t)\in B_0(K_0)$ for all $t\in [a,b]$. By (G3) applied to $U=B_0(K_0)$ we have that there exists $H_U\in\mathbb{R}$ such that for all $t \in [a, b]$
\begin{eqnarray}
\frac{d}{dt} \| z(t) - \tilde z(t) \|^2 &=& \frac{d}{dt} \left\| \int_a^t (g(s,z(s)) - g(s,\tilde z(s))) ds  \right\|^2 \notag\\
&\leq& 2 \langle z(t) - \tilde z(t), g(t, z(t)) - g(t, \tilde z(t))  \rangle \notag\\
&\leq& 2 H_U \|z(t) - \tilde z(t)\|^2 \leq 2H_U^+ \|z(t) - \tilde z(t)\|^2.
\end{eqnarray}
Hence, the $C^1$-function $[a,b]\ni t\to \varphi(t) =  \| z(t)-\tilde z(t)\|^2\in [0, +\infty)$, where $\varphi(a)=0$, satisfies the following integral inequality
\begin{displaymath}
    \varphi(t) \leq 2 H_U^+ \int_a^t \varphi(s)ds.
\end{displaymath}
From the  Gronwall's lemma we therefore get that $\varphi(t) = 0$ for all $t \in [a,b]$, which in turn implies that  $z(t) = \tilde z(t)$ for all $t\in [a,b]$.
\end{proof}

\begin{rem}
An existence and a uniqueness of the solution of \eqref{ode_gen} assuming (G1)-(G3) follow also from Theorem 3.27. page 182 in \cite{ParRas2014}.
\end{rem}
\begin{fact}
\label{ex_g_g1g4}
  The function \eqref{exampl_g1} satisfies the assumptions (G1)-(G4) with $\alpha=\alpha$ and $\beta=\min\{\varrho_1,\varrho_2\}$.
\end{fact}
\begin{proof}
We have that $g(t,x)=\sum\limits_{i=1}^3 h_i(t,x)$ where  $h_1(t,x)=A|t|^{\alpha}x\sin(x^2+1)$ and $h_i(t,x)=B_i\sgn(x) |x|^{\varrho_i}$ for $i=2,3$, $(t,x)\in [a,b]\times\mathbb{R}$. Of course $h_i\in C([a,b]\times\mathbb{R};\mathbb{R})$ for $i=1,2,3$ and therefore $g$ satisfies (G1). Moreover, for all $(t,x)\in [a,b]\times\mathbb{R}$
\begin{equation}
    |g(t,x)|\leq |A| D_1 |x|+(|B_1|+|B_2|)(1+|x|)\leq \Bigl(|A| D_1 +|B_1|+|B_2|\Bigr)(1+|x|),
\end{equation}
where $D_1=\max\limits_{a\leq t\leq b}|t|^{\alpha}$. Therefore, the function $g$ satisfies (G2). Next, for all $t_1,t_2\in [a,b]$, $x_1,x_2\in\mathbb{R}$ we have that
\begin{equation}
    |h_1(t_1,x_1)-h_1(t_2,x_2)|\leq |h_1(t_1,x_1)-h_1(t_1,x_2)|+|h_1(t_1,x_2)-h_1(t_2,x_2)|,
\end{equation}
where
\begin{equation}
\label{h_1_loc_lip_1}
    |h_1(t_1,x_1)-h_1(t_1,x_2)|\leq \frac{3}{2}|A|\cdot D_1\cdot (1+x_1^2+x_2^2)\cdot |x_1-x_2|,
\end{equation}
and, by Lemma A1.3 in  \cite{CZPMPP2022},
\begin{equation}
\label{h_1_loc_lip_2}
    |h_1(t_1,x_2)-h_1(t_2,x_2)|\leq |A|\cdot |x_2|\cdot |t_1-t_2|^{\alpha}.
\end{equation}
This implies, in particular, that for all $t\in [a,b]$, $x_1,x_2\in \mathbb{R}$
\begin{eqnarray}
\label{OSL_h_1}
    &&(x_1-x_2)(h_1(t,x_1)-h_1(t,x_2))\leq |(x_1-x_2)(h_1(t,x_1)-h_1(t,x_2))|\notag\\
    &&= |x_1-x_2|\cdot |h_1(t,x_1)-h_1(t,x_2)|\leq \frac{3}{2}|A|\cdot D_1\cdot (1+x_1^2+x_2^2)\cdot |x_1-x_2|^2.
\end{eqnarray}
Since $B_1,B_2\leq 0$, from the proof of Fact A.4 in \cite{CZPMPP2022} we have that for all $t\in [a,b]$, $x_1,x_2\in\mathbb{R}$, $i=2,3$
\begin{equation}
\label{OSL_h_i}
    (x_1-x_2)(h_i(t,x_1)-h_i(t,x_2))\leq 0,
\end{equation}
and for all $t_1,t_2\in [a,b]$, $x_1,x_2\in\mathbb{R}$, $i=2,3$
\begin{equation}
\label{h_i_holder}
    |h_i(t_1,x_1)-h_i(t_2,x_2)|\leq 2|B_i|\cdot |x_1-x_2|^{\varrho_i}.
\end{equation}
Hence, by \eqref{OSL_h_1}, \eqref{OSL_h_i} we get for all $t\in [a,b]$, $x_1,x_2\in\mathbb{R}$
\begin{eqnarray}
    && (x_1-x_2)(g(t,x_1)-g(t,x_2))=(x_1-x_2)\cdot\sum\limits_{i=1}^3 (h_i(t,x_1)-h_i(t,x_2))\notag\\
    &&\leq (x_1-x_2)(h_1(t,x_1)-h_1(t,x_2))\leq  \frac{3}{2}|A|\cdot D_1\cdot (1+x_1^2+x_2^2)\cdot |x_1-x_2|^2.
\end{eqnarray}
Since for every $U\in\comp(\mathbb{R})$ there exists a ball $B_0(R)$, with the radius $R\in (0,+\infty)$, such that $U\subset B_0(R)$, for all $t\in [a,b]$, $x_1,x_2\in U$ we have
\begin{equation}
    (x_1-x_2)(g(t,x_1)-g(t,x_2))\leq  \frac{3}{2}|A|\cdot D_1\cdot (1+2R^2)\cdot |x_1-x_2|^2,
\end{equation}
and, by \eqref{h_1_loc_lip_1},\eqref{h_1_loc_lip_2}, \eqref{h_i_holder}, for all $t_1,t_2\in [a,b]$, $x_1,x_2\in\mathbb{R}$
\begin{eqnarray*}
    &&|g(t_1,x_1)-g(t_2,x_2)|\leq\sum\limits_{i=1}^3 |h_i(t_1,x_1)-h_i(t_2,x_2)|\notag\\
    &&\leq |A|\cdot R\cdot |t_1-t_2|^{\alpha}+\max\Bigl\{\frac{3}{2}|A|D_1(1+2R^2),2|B_1|,2|B_2|\Bigr\}\cdot 3 (1+2R)\cdot |x_1-x_2|^{\min\{\varrho_1,\varrho_2\}}.
\end{eqnarray*}
Hence, $g$ satisfies also (G3), (G4), and this ends the proof.
\end{proof}
\begin{fact}
\label{ex_multipl_g1g4}
  The function \eqref{exampl_g2} satisfies the assumptions (G1)-(G4) with $\alpha=\alpha$ and $\beta=\varrho$.
\end{fact}
\begin{proof}
Of course, $g \in C([a,b]\times\R; \R)$ and therefore $g$ satisfies (G1).
Furthermore, for all $(t,x) \in [a,b]\times\R$
\begin{align*}
    |g(t,x)| &  \leq D\sup_{t \in [a,b]} |\gamma(t)| (1 + |x|),
\end{align*}
hence $g$ satisfies (G2) with $K=D \sup\limits_{t \in [a,b]} |\gamma(t)|$. Moreover, $h$ is monotonically decreasing, thus for all $t\in [a,b]$ and $x,y \in \R$
\begin{align*}
    (x-y)\Bigl(g(t, x) - g(t, y)\Bigr) &= \gamma(t)(x-y)\Bigl(h(x)f(x) - h(y)f(y)\Bigr) \\
    & = \gamma(t)f(x)(x-y)(h(x) - h(y)) + \gamma(t)(x-y)h(y)(f(x) - f(y)) \\
    & \leq \gamma(t)(x-y)h(y)(f(x) - f(y)) 
     \leq |\gamma(t)(x-y)h(y)(f(x) - f(y))| \\
    & \leq L_f\sup_{t \in [a,b]}|\gamma(t)|\cdot (1+|y|) \cdot |x-y|^2 .
\end{align*}
For all $U \in \comp(\R)$ there exists a ball $B_0(R)$, $R>0$, such that $U \subset B_0(R)$ and for all $t \in [a,b]$, $x, y \in U$ we have
\begin{displaymath}
    (x-y)\Bigl(g(t, x) - g(t, y)\Bigr) \leq L_f\sup_{t \in [a,b]}|\gamma(t)|\cdot (1+R) \cdot |x-y|^2.
\end{displaymath}
Hence, $g$ satisfies the assumption (G3). Furthermore, for all $t_1, t_2 \in [a,b]$ and $y_1, y_2 \in \R$ we have
\begin{displaymath}
    |g(t_1, y_1) - g(t_2, y_2)| \leq |g(t_1, y_1) - g(t_2, y_1)| + |g(t_2, y_1) - g(t_2, y_2)|,
\end{displaymath}
where
\begin{equation}
    |g(t_1, y_1) - g(t_2, y_1)|  \leq L_{\gamma} D \cdot (1+|y_1|)\cdot  |t_1 - t_2|^\alpha,
\end{equation}
and
\begin{align*}
    |g(t_2, y_1)  - g(t_2, y_2)| 
    & \leq L_f\sup_{t \in [a,b]} |\gamma(t)| \cdot  (1+|y_1|)\cdot |y_1 - y_2| + 2D\sup_{t \in [a,b]} |\gamma(t)| \cdot   |y_1 - y_2|^\varrho \\
    &\leq  \sup_{t \in [a,b]} |\gamma(t)| \cdot \max\{L_f, D\} \cdot\Bigl[(1+|y_1|)\cdot (1+|y_1|+|y_2|)+2\Bigr]\cdot |y_1-y_2|^{\varrho}.
\end{align*}
Thus, since for all $U \in \comp(\R)$ there exists a ball $B_0(R)$, $R>0$, such that $U \subset B_0(R)$, for all  $t_1, t_2 \in [a,b]$, $y_1, y_2 \in U$ we have
\begin{eqnarray*}
    |g(t_1, y_1) &-& g(t_2, y_2)| \leq L_\gamma D (1+R)|t_1 - t_2|^\alpha\\ 
    && +\sup_{t \in [a,b]} |\gamma(t)| \cdot \max\{L_f, D\}\Bigl[ (1+R)(1+2R) + 2 \Bigr]|y_1 - y_2|^\varrho. 
\end{eqnarray*}
This implies that $g$ satisfies (G4) with $\alpha = \alpha$ and $\beta = \varrho$, and the proof is completed.
\end{proof}


\end{document}